\documentclass[amstex,12pt]{article}
\usepackage{mathrsfs}
\usepackage{bbm}
\usepackage{amsfonts}
\usepackage{amssymb,amsmath,amsthm}
\newtheorem{theorem}{Theorem}[section]

\newtheorem{lemma}{Lemma}[section]
\newtheorem{definition}{Definition}[section]
\newtheorem{remark}{Remark}[section]

\newcommand{\beq}{\begin{equation}}
\newcommand{\eeq}{\end{equation}}
\newcommand{\beqn}{\begin{eqnarray}}
\newcommand{\eeqn}{\end{eqnarray}}

\begin{document}
\title{New Quantitative
Deformation Lemma and New Mountain Pass Theorem}\author {Liang
Ding$^{1,\,3}$\thanks{Email: lovemathlovemath@126.com(Liang Ding),
 },
Fode Zhang$^{2}$\thanks
 {Email: lnsz-zfd@163.com(Fode Zhang)},\,\,and\, Shiqing Zhang$^{1}$\thanks{Corresponding author's email: zhangshiqing@scu.edu.cn(Shiqing Zhang),},\\
$^{1}$Department of Mathematics and Yangtze Center of Mathematics,\\
Sichuan University\\ Chengdu 610064, People's Republic of
China\\
 $^{2}$Department of Mathematics\\ Kunming University of Science and
Technology\\ Yunnan 650093,
 People's Republic of China\\
 $^{3}$Department of Basis Education\\
 Dehong Vocational College\\
 Mangshi, Yunnan, 678400\\
 People's Republic of China\\
}\date{} \maketitle{}   \vskip-9mm\noindent
 {\bf Abstract} In this paper, we obtain a new quantitative deformation Lemma so that we can
obtain more critical points, especially for supinf critical value $c_1$,
$x=\varphi^{-1}(c_1)$ is a new critical point. For $infmax$ critical value $c_2$, we can obtain two new critical points $x = 0$ (valley point) and $x = e$(peak point) ,comparing with Willem's variant of the mountain pass theorem     of Ambrosetti-Rabinowitz,in which  $\varphi(e)\leq\varphi(0)<c_2$, but in our new mountain pass theorem,
$ \varphi(e)=c_2$. 
 {\bf 2010 MR Subject Classification} 47H10, 47J30, 39A10.

\section{Introduction}

\setcounter{equation}{0}
 \indent \allowdisplaybreaks

 In 1973, Ambrosetti and Rabinowitz \cite{1} presented the famous Mountain Pass Theorem.
 Later, there were many variants and generalizations([2]-[14]).
 Specially, Willem [11] gave the Quantitative Deformation Lemma and
  the corresponding mountain pass theorem. It is well known that quantitative deformation lemma is
 a very powerful tool to obtain mountain pass theorem, and the
 mountain pass theorem has proved to be a power tool in many areas
 of analysis. But to our best knowledge, very few works have been
 done for quantitative deformation lemma or mountain pass theorem in
 the past thirty years.

In this paper, we extend the quantitative deformation lemma in [11]
 so that we can obtain more critical points, especially for supinf
critical value $c_1$, $x=\varphi^{-1}(c_1)$ is a new critical point.
Moreover,as an application of our deformation lemma,a new mountain
pass theorem is given. Comparing with the mountain pass type theorem in
[11], $ \varphi(e)\leq\varphi(0)<c_2$, but in our new mountain pass
theorem, $ \varphi(e)=c_2$, so that our new mountain pass theorem can not
be obtained by the quantitative deformation lemma in [11];besides,
in our theorem, if $\varphi$ satisfies $(PS)_{c_2}$
condition, we can obtain two new critical points $x = 0$ (valley
point) and $x = e$ (peak point).

The organization of  this paper is as following. In
section $2$, the quantitative deformation lemma in [11] and the corresponding
mountain pass theorem in [11] are given. In section $3$, on the
basis of the quantitative deformation lemma in [11], we prove
the new quantitative deformation lemma. In section $4$, as an application of our deformation lemma,
our new mountain pass theorem is given.

\section{Preliminaries}
\setcounter{equation}{0}
 \indent

 For convenience, we introduce the Quantitative Deformation Lemma in [11] and
 the corresponding Mountain Pass Type Theorem in [11] as the following:
  \begin{lemma}(Quantitative deformation lemma)
 Let $X$ be a Hilbert space, $\varphi\in C^{2}(X, \mathbb{R})$, $c\in \mathbb{R}$, $\varepsilon>0$. Assume that
  \[
  \big(\forall u\in\varphi^{-1}([c-2\varepsilon,
  c+2\varepsilon])\big):\|\varphi^{\prime}(u)\|\geq2\varepsilon.
  \]
  Then there exists $\eta\in$ $C(X, X)$, such that
  \begin{itemize}
  \item[$(a)$] $\eta(u)=u$, $\forall u\notin \varphi^{-1}\big([c-2\varepsilon,
  c+2\varepsilon]\big)$.
 \item[$(b)$]
 $\eta(\varphi^{c+\varepsilon})\subset\varphi^{c-\varepsilon}$,
 where $\varphi^{c-\varepsilon}:=\varphi^{-1}\big((-\infty, c-\varepsilon]\big)$.
\end{itemize}
  \end{lemma}

 \begin{theorem}(Mountain pass type theorem)
  Let $X$ be a Hilbert space, $\varphi\in C^{2}(X, \mathbb{R})$, $e\in X$ and $r>0$ be such that $\|e\|>r$ and
     \begin{eqnarray}\label{2.1}
      b:=\inf_{\|u\|=r}\varphi(u)>\varphi(0)\geq\varphi(e).
     \end{eqnarray}
     Then, for each $\varepsilon>0$, there exists $u\in X$ such that
     \begin{itemize}
      \item[$(i)$] $c-2\varepsilon\leq\varphi(u)\leq c+2\varepsilon$,
      \item[$(ii)$] $\|\varphi^{\prime}(u)\|<2\varepsilon$,
       \end{itemize}
       where
        \[
               c:=\inf_{\gamma\in\Gamma}\max_{t\in[0, 1]}\varphi\big(\gamma(t)\big)
        \]
        and
         \[
          \Gamma:=\{\gamma\in C\big([0, 1], X\big):\gamma(0)=0,
          \gamma(1)=e\}.
         \]
\end{theorem}

\begin{definition}
([14])Let $X$ be a Banach space, $\varphi \in C^{1}(X, \mathbb{R})$ and
$c\in \mathbb{R}$. The function $\varphi$ satisfies the $(PS)_{c}$
condition if any sequence $(u_{n}) \subset X$ such that
\[\varphi(u_{n})\rightarrow c,
\varphi^{\prime}(u_{n})\rightarrow 0\]
 has a convergent subsequence.
\end{definition}
\section{New Quantitative Deformation Lemma}
\setcounter{equation}{0}
 \indent

\begin{theorem}\label{3.1} Let $X$ be a Hilbert space, $\varphi\in C^{2}(X, \mathbb{R})$, $c\in \mathbb{R}$,
$\varepsilon>0$. Assume that
\begin{eqnarray*}
\big(\forall u\in\varphi^{-1}([c-2\varepsilon,
c+2\varepsilon])\big):\|\varphi^{\prime}(u)\|\geq2\varepsilon.
\end{eqnarray*}
Then there exists $\eta\in$ $C(X, X)$, such that
\begin{itemize}
\item[$(a\,^{\prime})$] $\eta(u)=u$, $\forall u\notin \varphi^{-1}\big([c-2\varepsilon,
c+2\varepsilon]\big)\backslash D$, where
$D\subseteq\varphi^{-1}\big([c-0.5\varepsilon, c+\varepsilon]\big)$.
\item[$(b\,^{\prime})$] $\eta\big(\varphi^{-1}[c-\varepsilon, c-0.6\varepsilon]\big)\subset\varphi_{*}^{c+\varepsilon}$, where
$\varphi_{*}^{c+\varepsilon}$ denotes
$\varphi^{-1}\big([c+\varepsilon, +\infty)\big)$.
\item[$(c\,^{\prime})$] $\eta\big(\varphi^{-1}[c+0.6\varepsilon, c+\varepsilon]\big)\subset\varphi^{c-\varepsilon}$, where
$\varphi^{c-\varepsilon}$ denotes $\varphi^{-1}\big((-\infty,
c-\varepsilon]\big)$.
\end{itemize}
\end{theorem}
\begin{proof}
Let us define
\begin{eqnarray*}
&A&:=\varphi^{-1}\big([c-2\varepsilon, c+2\varepsilon]\big)\backslash D,\\
&B&:=\varphi^{-1}\big([c-\varepsilon, c-0.6\varepsilon]\big),\\
&C&:=\varphi^{-1}\big([c+0.6\varepsilon, c+\varepsilon]\big),\\
&\psi(u)&:=\frac{[dist(u, C)-dist(u, B)]dist(u, X\backslash
A)}{[dist(u, C)+dist(u, B)]dist(u, X\backslash A)+dist(u, B)dist(u,
C)},
\end{eqnarray*}
so that $\psi$ is locally Lipschitz continuous, $\psi=1$ on $B$,
$\psi=-1$ on $C$ and
$\psi=0$ on $X\backslash$ $A$.\\
Let us also define the locally Lipschitz continuous vector field
\begin{eqnarray*}
f(u)&:=&\psi(u)\|\nabla\varphi(u)\|^{-2}\nabla\varphi(u), \quad \ \  u\in A,\\
&:=&0, \quad \ \  u\in X\backslash A.
\end{eqnarray*}
It is clear that $\|f(u)\|\leq (2\varepsilon)^{-1}$ on $X$. For each
$u\in$ $X$, the Cauchy problem
\begin{eqnarray*}
\frac{d}{dt}\sigma(t, u)&=&f\big(\sigma(t, u)\big),\\
\sigma(0, u)&=&u,
\end{eqnarray*}
has a unique solution $\sigma(\cdot, u)$ defined on $\mathbb{R}$.
Moreover, $\sigma$ is continuous on $\mathbb{R}\times X$(see e.g.
\cite{10}). The map $\eta$ defined on $X$ by
$\eta(u):=\sigma(2\varepsilon, u)$ satisfies $(a\,^{\prime})$. Since
\begin{eqnarray}
\frac{d}{dt}\varphi\big(\sigma(t,
u)\big)&=&\bigg(\nabla\varphi\big(\sigma(t, u)\big),
\frac{d}{dt}\sigma(t, u)\bigg)\nonumber\\
&=&\bigg(\nabla\varphi\big(\sigma(t, u)\big), f\big(\sigma(t,
u)\big)\bigg)\nonumber\\
&=&\psi\big(\sigma(t, u)\big)\label{3.1},
\end{eqnarray}
 If
\begin{eqnarray*}
\sigma(t, u)\in\varphi^{-1}\big([c-\varepsilon,
c-0.6\varepsilon]\big)=B, \quad \ \ \forall t\in[0, 2\varepsilon],
\end{eqnarray*}
then \[ \psi(\sigma(t, u))=1.
\]
So, we obtain from \eqref{3.1},
\begin{eqnarray*}
\varphi\big(\sigma(2\varepsilon,
u)\big)&=&\varphi(u)+\int_{0}^{2\varepsilon}\frac{d}{dt}\varphi\big(\sigma(t,
u)\big)dt\\
&=&\varphi(u)+\int_{0}^{2\varepsilon}\psi\big(\sigma(t, u)\big)dt\\
&\geq&c-\varepsilon+2\varepsilon=c+\varepsilon,
\end{eqnarray*}
and $(b\,^{\prime})$ is also satisfied. \\
Finally, similar to prove $(b\,^{\prime})$, we prove $(c\,^{\prime})$.\\
If
\begin{eqnarray*}
\sigma(t, u)\in\varphi^{-1}\big([c+0.6\varepsilon,
c+\varepsilon]\big)=C, \quad \ \ \forall t\in[0, 2\varepsilon],
\end{eqnarray*}
then \[
 \psi(\sigma(t, u))=-1.
\]
So, we obtain from \eqref{3.1},
\begin{eqnarray*}
\varphi\big(\sigma(2\varepsilon,
u)\big)&=&\varphi(u)+\int_{0}^{2\varepsilon}\frac{d}{dt}\varphi\big(\sigma(t,
u)\big)dt\\
&=&\varphi(u)+\int_{0}^{2\varepsilon}\psi\big(\sigma(t, u)\big)dt\\
&\leq&c+\varepsilon-2\varepsilon=c-\varepsilon,
\end{eqnarray*}
and $(c\,^{\prime})$ is also satisfied.
\end{proof}
\begin{remark}\label{3.1}
By Theorem $\ref{3.1}$, we get more critical points than the
Quantitative Deformation Lemma in [11]. All the domain $D$
in Theorem $\ref{3.1}$, especially for $supinf$ critical value $c$,
$x=\varphi^{-1}(c)$ are all new critical points.
\end{remark}
\begin{remark}
In Lemma $2.1$, there are two conclusions, but in Theorem
$\ref{3.1}$, there are three conclusions.
\end{remark}
 \section{An Example (New Mountain Pass Theorem)}
   \setcounter{equation}{0} \noindent

 Let $X$ be a Hilbert space,
$\varphi\in C^{2}(X, \mathbb{R})$, $e\in X$ and $r>0$ be such that
$\|e\|>r$ and
\[
 \varphi(0)=c_1, \quad \ \ \varphi(e)=c_2, \quad \ \ c_1\neq c_2,
 \]
  and
\[
          c_1:=\sup_{\gamma\in\Gamma}\min_{t\in[0, 1]}\varphi\big(\gamma(t)\big)
   , \,\,\,\,c_2:=\inf_{\gamma\in\Gamma}\max_{t\in[0,
   1]}\varphi\big(\gamma(t)\big),
   \]
   where
    \[
     \Gamma:=\{\gamma\in C\big([0, 1], X\big):\gamma(\frac{1}{4})=0, \gamma(\frac{1}{2})=e\}
    .\]
 Then, for each $\varepsilon>0$, there exists $u^{\ast}\in X$ and $u^{\triangle}\in X$ such that
\begin{itemize}
\item[$(\mathrm{I})$] $c_1-2\varepsilon\leq\varphi(u^{\ast})\leq c_1+2\varepsilon$,
 \end{itemize}
 \begin{itemize}
 \item[$(\mathrm{II})$] $\|\varphi^{\prime}(u^{\ast})\|<2\varepsilon$.
   \end{itemize}
 \begin{itemize}
   \item[$(\mathrm{III})$] $c_2-2\varepsilon\leq\varphi(u^{\triangle})\leq c_2+2\varepsilon$,
 \end{itemize}
 \begin{itemize}
 \item[$(\mathrm{IV})$] $\|\varphi^{\prime}(u^{\triangle})\|<2\varepsilon$.
 \end{itemize}
\begin{proof}
Obviously, for each $\varepsilon>0$, (I) and (III) are easy to
get. Next, we prove (II)
and (IV).\\
Suppose that  at least one of (II) and (IV) is not true. Then, we
can get the contradiction:

Case 1. We assume that (II) is not true. It means that there exists
$\varepsilon$ such that
\[
 \|\varphi^{\prime}(u^{\ast})\|\geq2\varepsilon.
\]
From
\[
          c_1:=\sup_{\gamma\in\Gamma}\min_{t\in[0, 1]}\varphi\big(\gamma(t)\big)
   , \,\,\,\,c_2:=\inf_{\gamma\in\Gamma}\max_{t\in[0,
   1]}\varphi\big(\gamma(t)\big),
   \]
   and $c_1\neq c_2$,
 we get $c_1< c_2$ or $c_1> c_2$.
Then, Case 1 can be divided into two parts.

 Firstly, when  $c_1< c_2$, let
$\varepsilon_{1}=\min\{\frac{c_2-c_1}{4},\,\varepsilon\}$. It is
clear that
\[
 \|\varphi^{\prime}(u^{\ast})\|\geq2\varepsilon_{1}.
\]
and for $\varepsilon_{1}$, (I) is still easy to get.

 From $\varepsilon_{1}=\min\{\frac{c_2-c_1}{4},\,\varepsilon\}$,
we obtain
\[
c_1+2\varepsilon_{1}\leq c_1+2\times\frac{c_2-c_1}{4}
=c_1+\frac{c_2-c_1}{2}=\frac{c_2}{2}+\frac{c_1}{2}<c_2.
\]
It means that
\[
c_2>c_1+2\varepsilon_{1}.
\]
 In Theorem $\ref{3.1}$, we
 can take $D = \{u\in X\mid\varphi(u)=c_1\}$. Consider
$\beta=\eta\circ\gamma$, where $\eta$ is given by Theorem
 $\ref{3.1}$. Using $(a\,^{\prime})$, we have,
 \begin{eqnarray*}
\beta(\frac{1}{4})&=&\eta\big(\gamma(\frac{1}{4})\big)=\eta(0)=0.\\
\beta(\frac{1}{2})&=&\eta\big(\gamma(\frac{1}{2})\big)=\eta(e)=e,
     \end{eqnarray*}
so that $\beta\in\Gamma$.
  From
   \[
             c_1:=\sup_{\gamma\in\Gamma}\min_{t\in[0,
             1]}\varphi\big(\gamma(t)\big),
      \]
   there exist $\gamma\in \Gamma$ and $\varepsilon_{2}>0$ such that
  \[
   c_1-\varepsilon_{2}\leq\min_{t\in[0, 1]}\varphi\big(\gamma(t)\big)\leq c_1-0.6\varepsilon_{2}.
   \]
   Then, from $(b\,^{\prime})$, we have
  \[
    \min_{t\in[0, 1]}\varphi\bigg(\eta\big(\gamma(t)\big)\bigg)\geq
    c_1+\varepsilon_{2}.
  \]
  It means that
   \[
       \min_{t\in [0, 1]}\varphi\big(\beta(t)\big)\geq
    c_1+\varepsilon_{2}.
     \]
    So
\[
   c_1+\varepsilon_{2}\leq\min_{t\in [0, 1]}\varphi\big(\beta(t)\big)\leq c_1.
 \]
 This is a contradiction.
 Therefore, (II) is true.

Secondly, when  $c_1> c_2$, let
$\varepsilon_{1}=\min\{\frac{c_1-c_2}{4},\,\varepsilon\}$. It is
clear that
\[
 \|\varphi^{\prime}(u^{\ast})\|\geq2\varepsilon_{1},
\]
and for $\varepsilon_{1}$, (I) is still easy to get.

 From $\varepsilon_{1}=\min\{\frac{c_1-c_2}{4},\,\varepsilon\}$,
we obtain
\[
c_1-2\varepsilon_{1}\geq c_1-2\times\frac{c_1-c_2}{4}
=\frac{c_1}{2}+\frac{c_2}{2}>c_2.
\]
It means that
\[
c_2<c_1-2\varepsilon_{1}.
\]
 In Theorem $\ref{3.1}$, we
 can take $D = \{u\in X\mid\varphi(u)=c_1\}$. Consider
$\beta=\eta\circ\gamma$, where $\eta$ is given by Theorem
 $\ref{3.1}$. Using $(a\,^{\prime})$, we have,
 \begin{eqnarray*}
\beta(\frac{1}{4})&=&\eta\big(\gamma(\frac{1}{4})\big)=\eta(0)=0.\\
\beta(\frac{1}{2})&=&\eta\big(\gamma(\frac{1}{2})\big)=\eta(e)=e,
     \end{eqnarray*}
so that $\beta\in\Gamma$.
  From
   \[
             c_1:=\sup_{\gamma\in\Gamma}\min_{t\in[0,
             1]}\varphi\big(\gamma(t)\big),
      \]
   there exist $\gamma\in \Gamma$ and $\varepsilon_{2}>0$ such that
  \[
   c_1-\varepsilon_{2}\leq\min_{t\in[0, 1]}\varphi\big(\gamma(t)\big)\leq c_1-0.6\varepsilon_{2}.
   \]
   Then, from $(b\,^{\prime})$, we have
  \[
    \min_{t\in[0, 1]}\varphi\bigg(\eta\big(\gamma(t)\big)\bigg)\geq
    c_1+\varepsilon_{2}.
  \]
  It means that
   \[
       \min_{t\in [0, 1]}\varphi\big(\beta(t)\big)\geq
    c_1+\varepsilon_{2}.
     \]
    So
\[
   c_1+\varepsilon_{2}\leq\min_{t\in [0, 1]}\varphi\big(\beta(t)\big)\leq c_1.
 \]
 This is a contradiction.
 Therefore, (II) is true.

 Case 2. We assume that (IV) is not true. It means that there exists
$\varepsilon$ such that
\[
\|\varphi^{\prime}(u^{\triangle})\|\geq2\varepsilon.
\]
From
\[
          c_1:=\sup_{\gamma\in\Gamma}\min_{t\in[0, 1]}\varphi\big(\gamma(t)\big)
   , \,\,\,\,c_2:=\inf_{\gamma\in\Gamma}\max_{t\in[0,
   1]}\varphi\big(\gamma(t)\big),
   \]
   and $c_1\neq c_2$
 we get $c_1< c_2$ or $c_1> c_2$. Then, Case 2 can be divided into two parts.

Firstly, when  $c_1< c_2$, let
$\varepsilon_{1}=\min\{\frac{c_2-c_1}{4},\,\varepsilon\}$. It is
clear that
\[
 \|\varphi^{\prime}(u^{\triangle})\|\geq2\varepsilon_{1}.
\]
and for $\varepsilon_{1}$, (III) is still easy to get.

 From $\varepsilon_{1}=\min\{\frac{c_2-c_1}{4},\,\varepsilon\}$,
we obtain
\[
c_2-2\varepsilon_{1}\geq c_2-2\times\frac{c_2-c_1}{4}
=c_2-\frac{c_2-c_1}{2}=\frac{c_2}{2}+\frac{c_1}{2}>c_1.
\]
It means that
\[
c_1<c_2-2\varepsilon_{1}.
\]

 In Theorem $\ref{3.1}$, we
 can take $D = \{u\in X\mid\varphi(u)=c_2\}$. Consider
$\beta=\eta\circ\gamma$, where $\eta$ is given by Theorem
 $\ref{3.1}$. Using $(a\,^{\prime})$, we have,
 \begin{eqnarray*}
\beta(\frac{1}{4})&=&\eta\big(\gamma(\frac{1}{4})\big)=\eta(0)=0.\\
\beta(\frac{1}{2})&=&\eta\big(\gamma(\frac{1}{2})\big)=\eta(e)=e,
\end{eqnarray*}
so that $\beta\in\Gamma$.
    From
   \[
             c_2:=\inf_{\gamma\in\Gamma}\max_{t\in[0,
             1]}\varphi\big(\gamma(t)\big),
      \]
   there exist $\gamma\in \Gamma$ and $\varepsilon_{3}>0$ such that
  \[
   c_2+0.6\varepsilon_{3}\leq\max_{t\in[0, 1]}\varphi\big(\gamma(t)\big)\leq c_2+\varepsilon_{3}.
   \]
   Then, from $(c_1\,^{\prime})$, we have
  \[
    \max_{t\in[0, 1]}\varphi\bigg(\eta\big(\gamma(t)\big)\bigg)\leq
    c_2-\varepsilon_{3}.
  \]
  It means that
   \[
       \max_{t\in [0, 1]}\varphi\big(\beta(t)\big)\leq
       c_2-\varepsilon_{3}.
     \]
    So
\[
   c_2\leq\max_{t\in [0, 1]}\varphi\big(\beta(t)\big)\leq
   c_2-\varepsilon_{3}.
 \]
 This is a contradiction. Therefore, (IV) is true.

Secondly, when  $c_2< c_1$, let
$\varepsilon_{1}=\min\{\frac{c_1-c_2}{4},\,\varepsilon\}$ and take
$D = \{u\in X\mid\varphi(u)=c_2\}$, the rest of the proof is similar
to the first part of Case 2. Therefore, (IV) is true.

From Case 1 and Case 2, our new mountain pass theorem is proved.

\end{proof}

\begin{remark}\label{4.1}
 In Theorem 2.1 (Mountain pass theorem), $c$ is defined as
\[
c:=\inf_{\gamma\in\Gamma}\max_{t\in[0, 1]}\varphi\big(\gamma(t)\big)
\]
where
\[
 \Gamma:=\{\gamma\in C\big([0, 1], X\big):\gamma(0)=0,
          \gamma(1)=e\}.
\]
But in our new mountain pass theorem, $c_1$ and $c_2$ are defined as
\[
c_1:=\sup_{\gamma\in\Gamma}\min_{t\in[0,
1]}\varphi\big(\gamma(t)\big),\quad \ \
c_2:=\inf_{\gamma\in\Gamma}\max_{t\in[0,
1]}\varphi\big(\gamma(t)\big),
\]
where
\[
\Gamma:=\{\gamma\in C\big([0, 1], X\big):\gamma(\frac{1}{4})=0,
          \gamma(\frac{1}{2})=e\}.
\]
\end{remark}
\begin{remark}\label{4.3}
 In fact, in Theorem 2.1 (Mountain pass theorem),
 \[
  c_2>\varphi(0)\geq\varphi(e).
  \]
  But in our new mountain pass theorem,
\[
 \varphi(0)=c_1, \quad \ \ \varphi(e)=c_2, \quad \ \ c_1\neq c_2.
 \]
 and in the proof of our new mountain pass theorem, we take
$D = \{u\in X\mid\varphi(u)=c_1\}$ in Case 1, and take $D = \{u\in
X\mid\varphi(u)=c_2\}$ in Case 2.
\end{remark}
\begin{remark}\label{4.4}
In the example, if we do not use our Theorem $\ref{3.1}$ (New
quantitative deformation lemma), we can not obtain
\[
\beta(\frac{1}{4})=\eta\big(\gamma(\frac{1}{4})\big)=\eta(0)=0.
\]
Moreover, we can not obtain $\beta\in\Gamma$.
\end{remark}

\begin{remark}\label{4.5}
An interesting point in the example is that, if $\varphi$ satisfies
$(PS)_{c}$ condition, it is easy to obtain two new critical points
$x = 0$ and $x = e$ which have not been obtained before.
\end{remark}


\begin{thebibliography}{10}

\bibitem{1} A. Ambrosetti and P.H. Rabinowitz, Dual variational
methods in critical point theory and applications, J. Funct. Anal.
14 (1973), 349-381.
\bibitem{2} G. Barletta and S.A. Marano, Some remarks on critical point theory for locally Lipschitz functions,
Glasgow Math. J. 45 (2003), 131-141.
\bibitem{3} H. Brezis, J.M. Coron and L. Nirenberg, Free
vibrations for a nonlinear wave equation and theorem of P.
Rabinowitz, Comm. Pure Appl. Math, 33 (1980), 667-684.
\bibitem{4} K.C.,Chang, Infinite dimensional Morse theory,Birkh\"{a}user, 1993.
\bibitem{5} N. Ghoussoub, Duality and Perturbation Methods in Critical Point Theory, Cambridge Tracts in Math.
107, Cambridge Univ. Press, Cambridge, 1993.
\bibitem{7}H. Hofer, A geometric description of the neighbourhood of a critical point given by the Mountain Pass Theorem, J. London Math. Soc. 31 (1985), 566-570.
\bibitem{8}I. Peral, Beyond the mountain pass: some applications. Adv. Nonlinear Stud. 12 (2012), no. 4, 819-850
\bibitem{9}P. Pucci and J. Serrin, A Mountain Pass Theorem, J. Differential Equations 60 (1985), 142-149.
\bibitem{10}P. Pucci and J. Serrin, Extensions of the Mountain Pass Theorem, J. Funct. Anal. 59 (1984), 185-210.
\bibitem{11}P. Pucci, J. Serrin, The structure of the critical set in the mountain pass theorem, Trans. Amer. Math. Soc. 299, (1987), no. 1, 115-132.``
\bibitem{6} M. Willem, Minimax Theorems, Birkh\"{a}user,Boston, 1996.
\bibitem{7} R. Livrea and S.A. Marano, Existence and classification of critical points for nondifferentiable functions, Adv. Differential Equations 9 (2004), 961-978.
\bibitem{8} S.A. Marano and D. Motreanu, A deformation theorem and some critical points results for non-
differentiable functions, Topol. Methods Nonlinear Anal. 22 (2003),
139-158.
\bibitem{9} P.H. Rabinowitz, Minimax Methods in Critical Point Theory with Applications to Differential Equations,
CBMS Reg. Conf. Ser. Math. 65, Amer. Math. Soc., Providence, RI,
1986.
\bibitem{10} Schwartz L., Cours $d^{,}$analyse, Hermann, Paris,
1991-1994.

\end{thebibliography}
\end{document}